\documentclass[11pt,twoside,a4paper]{article}
\usepackage{color}
\usepackage{amsmath}
\usepackage{amssymb}
\usepackage{amsthm}
\usepackage{latexsym}
\usepackage{stmaryrd}
\usepackage{shuffle}
\usepackage{tikz,tkz-tab, tikz-cd}

\newcommand{\pai}{\left(}
\newcommand{\pad}{\right)}

\tikzset{
  ncone/.pic={
	\draw (0,0)--(0,0.2);
  }
}

\tikzset{
  nctwo/.pic={
    \draw (0,0)--(0,0.2);
	\draw (0.1,0)--(0.1,0.2);
  }
}

\tikzset{
  nctwoW/.pic={
    \draw (0,0.2)--(0,0)--(0.1,0)--(0.1,0.2);
  }
}

\tikzset{
  nctwoWW/.pic={
    \draw (0,0.2)--(0,0)--(0.2,0)--(0.2,0.2);
  }
}

\tikzset{
  ncthreeWW/.pic={
    \draw (0,0.2)--(0,0)--(0.3,0)--(0.3,0.2);
	\draw (0.2,0)--(0.2,0.2);
  }
}

\tikzset{
  ncthree/.pic={
    \draw (0,0)--(0,0.2);
	\draw (0.1,0)--(0.1,0.2);
	\draw (0.2,0)--(0.2,0.2);
  }
}

\tikzset{
  ncthreeW/.pic={
    \draw (0,0.2)--(0,0)--(0.2,0)--(0.2,0.2);
	\draw (0.1,0)--(0.1,0.2);
  }
}

\tikzset{
  ncfour/.pic={
    \draw (0,0)--(0,0.2);
	\draw (0.1,0)--(0.1,0.2);
	\draw (0.2,0)--(0.2,0.2);
	\draw (0.3,0)--(0.3,0.2);
  }
}
\tikzset{
  ncfive/.pic={
    \draw (0,0)--(0,0.2);
	\draw (0.1,0)--(0.1,0.2);
	\draw (0.2,0)--(0.2,0.2);
	\draw (0.3,0)--(0.3,0.2);
	\draw (0.4,0)--(0.4,0.2);
  }
}

\tikzset{
  ncfourW/.pic={
    \draw (0,0.2)--(0,0)--(0.3,0)--(0.3,0.2);
	\draw (0.1,0)--(0.1,0.2);
	\draw (0.2,0)--(0.2,0.2);
  }
}
\tikzset{
  ncfiveW/.pic={
    \draw (0,0.2)--(0,0)--(0.4,0)--(0.4,0.2);
	\draw (0.1,0)--(0.1,0.2);
	\draw (0.2,0)--(0.2,0.2);
	\draw (0.3,0)--(0.3,0.2);
  }
}

\tikzset{
  nconeinsidetwoWW/.pic={
    \path (0,0) pic {nctwoWW}; \path (0.1,0.1) pic {ncone};
  }
}

\tikzset{
  nconeinsidethreeWW/.pic={
    \path (0,0) pic {ncthreeWW}; \path (0.1,0.1) pic {ncone};
  }
}
\tikzset{
  nconeinsidethreerightWW/.pic={
    \draw (0,0.2)--(0,0)--(0.3,0)--(0.3,0.2);
    \draw (0.1,0)--(0.1,0.2); \draw (0.2,0.1)--(0.2,0.3);
  }
}
\tikzset{
  nconeinsidethreeleftWW/.pic={
    \draw (0,0.2)--(0,0)--(0.3,0)--(0.3,0.2);
    \draw (0.2,0)--(0.2,0.2); \draw (0.1,0.1)--(0.1,0.3);
  }
}
\tikzset{
  nctwoWWW/.pic={
    \draw (0,0.2)--(0,0)--(0.3,0)--(0.3,0.2);
  }
}

\tikzset{
  nconeoneinsidetwoWW/.pic={
	\path (0,0) pic {ncone};
    \path (0.1,0) pic {nctwoWW}; 
	\path (0.2,0.1) pic {ncone};
  }
}

\usepackage[T1]{fontenc}   
\usepackage[english]{babel}
\theoremstyle{definition}
\newtheorem{defi}{\indent Definition}[section]
\newtheorem{rem}[defi]{\indent Remark}

\theoremstyle{theorem}
\newtheorem{lemma}[defi]{\indent Lemma}
\newtheorem{cor}[defi]{\indent Corollary}
\newtheorem{theo}[defi]{\indent Theorem}
\newtheorem{prop}[defi]{\indent Proposition}

\newcommand{\frontstick}{\,\raisebox{-1pt}{\begin{tikzpicture}
\draw [line width=1pt,] (0,0)--(0,0.25);
\end{tikzpicture}}\kern+2pt}

\newcommand{\NC}{\operatorname{NC}}

\usepackage{forest}
\forestset{
  decor/.style = {
    label/.expanded = {[inner sep = 0.2ex, font=\unexpanded{\tiny}]right:{$#1$}}
  },
  root/.style = {minimum size = 0.1ex},
  decorated/.style = {
    for tree = {
      circle, fill, inner sep = 0.3ex, minimum size = 1.ex,
      grow' = south, l = 0, l sep = 1.2ex, s sep = 0.7em,
      fit = tight, parent anchor = center, child anchor = center,
      delay = {decor/.option = content, content =}
    }
  },
  default preamble = {decorated, root},
  begin draw/.code={\begin{tikzpicture}[baseline={([yshift=-0.5ex]current bounding box.center)}]},
}

\definecolor{red}{rgb}{1.,0.,0.}
\definecolor{green}{rgb}{0.,1.,0.}
\definecolor{blue}{rgb}{0.,0.,1.}
\definecolor{orange}{rgb}{1.,0.8431372549019608,0.}

\title{Cumulant-cumulant relations in free probability theory\\ from Magnus' expansion}

\usepackage{calc}
\newcounter{PartitionDepth}
\newcounter{PartitionLength}

\begin{document}

\author{
A.~Celestino$\!\!\phantom{i}^{\ast}$\!\! , 
K.~Ebrahimi-Fard\footnote{Department of Mathematical Sciences, Norwegian University of Science and Technology (NTNU), 7491 Trondheim, Norway. \texttt{adrian.celestino@ntnu.no}, \texttt{kurusch.ebrahimi-fard@ntnu.no}}\! , 
F.~Patras\footnote{Universit\'e de Nice, Laboratoire J.-A.~Dieudonn\'e, UMR 7351, CNRS, Parc Valrose, 06108 Nice Cedex 02, France. \texttt{patras@math.unice.fr}}\! , 
D.~Perales Anaya\footnote{Department of Pure Mathematics, University of Waterloo, Ontario, Canada. \texttt{dperales@uwaterloo.ca} }}

\maketitle

\begin{abstract}
Relations between moments and cumulants play a central role in both classical and non-commutative probability theory. The latter allows for several distinct families of cumulants corresponding to different types of independences: free, Boolean and monotone. Relations among those cumulants have been studied recently. In this work we focus on the problem of expressing with a closed formula multivariate monotone cumulants in terms of free and Boolean cumulants. In the process we introduce various constructions and statistics on non-crossing partitions. Our approach is based on a pre-Lie algebra structure on cumulant functionals. Relations among cumulants are described in terms of the pre-Lie Magnus expansion combined with results on the continuous Baker--Campbell--Hausdorff formula due to A.~Murua.
\end{abstract}


\begin{quote}
{\footnotesize{\bf Keywords:} monotone cumulants; free cumulants; Boolean cumulants; irreducible non-crossing partitions; quasi-monotone partitions; pre-Lie algebra; pre-Lie Magnus expansion; rooted trees.}\\
{\footnotesize{\bf MSC Classification}: 16T30; 17A30; 46L53; 46L54}
\end{quote}

  
\tableofcontents


\section{Introduction}
\label{sec:intro}

Relations between moments and cumulants have been intensively studied in both classical and non-commutative probability theory. The latter, contrary to the classical theory, allows for several distinct families of cumulants corresponding to different types of independences: free, Boolean and monotone. It is natural to compare those cumulants. Arizmendi et al.~\cite{arizmendi_15} have studied in detail relations among classical, free, Boolean and monotone cumulants using the common approach based on M\"obius inversion in the various corresponding lattices of set partitions and other combinatorial and algebraic techniques.

In a series of recent papers, two of us developed an alternative approach based on non-commutative shuffle algebra combined with group and Lie algebra theory \cite{ebrahimipatras_15,ebrahimipatras_16,ebrahimipatras_18,ebrahimipatras_19,ebrahimipatras_20}. More precisely, these works explore the group of characters and its corresponding Lie algebra of infinitesimal characters on a non-commutative combinatorial Hopf algebra $H$ constructed from the data provided by a non-commutative probability space $(A,\varphi)$. Concretely, $H$ is defined to be the double tensor algebra over the latter. In this context, the linear functional $\varphi\colon A \to \mathbb{K}$, which defines the moments of non-commutative random variables, is interpreted as an algebra map from $H$ to the ground field $\mathbb{K}$ (usually the complex numbers). The various families of cumulants (free, Boolean and monotone), usually seen as multilinear functions over $(A,\varphi)$, enter naturally this picture as infinitesimal characters, i.e., linear forms on $H$ that vanish on the algebra unit as well as on products of non-unital elements in the augmentation ideal of $H$. This approach leads in particular to the definition of a pre-Lie algebra structure on the space of infinitesimal characters, which results naturally from the formalism of shuffle algebras. 

\smallskip

The pivotal aim of the work at hand is to employ the aforementioned pre-Lie algebra structure on the level of cumulants, now seen as multilinear maps on the non-commutative probability space. This approach is self-contained and does not require using shuffle or Hopf algebras. In fact, it will be shown that the notion of pre-Lie algebra provides an adequate setting for a concise theoretical presentation of multivariate cumulant-cumulant relations. In particular, it is tailored to describe precisely multivariate monotone-free and monotone-Boolean cumulant-cumulant relations in terms of the pre-Lie Magnus map and its inverse, the so-called pre-Lie exponential. Both of them are familiar in various domains of applied mathematics such as numerical analysis of differential equations and geometric control theory \cite{AS_04,blanes_09,iserles_00}. This is underlined by the fact that we use combinatorial constructions and formulas introduced by A.~Murua \cite{murua_07} in the study of the expansion of the continuous Baker--Campbell--Hausdorff formula in a Hall basis. In the process, still inspired by Murua's seminal work, we introduce operations on non-crossing partitions together with the notion of quasi-monotone partitions. These turn out to be related to other partitions previously considered in the context of free probability, see Remark \ref{rem.quasi.colored} below. As we are interested in obtaining explicit formulas, we enhance this new point of view by focusing on the combinatorics of non-planar rooted forests naturally associated to non-crossing partitions. 

\smallskip

The paper is organised as follows. Section \ref{sec:ncpart} briefly recalls the basics on non-crossing partitions including the link to non-planar rooted trees encoding the hierarchy of nested blocks in a non-crossing partition as well as tree factorials. In Section \ref{sec:preLie} we give a self-contained presentation of the pre-Lie structure on the space of cumulants -- the space $\mathfrak g$ of multilinear maps taking values in the complex numbers, defined on a non-commutative probability space. We recall the theoretical results on the pre-Lie Magnus map and its inverse relevant to relations between cumulants.
In Section \ref{sec:Fla1}, we compute explicitly iterated pre-Lie products in $\mathfrak g$ and, as a first application, show how the formula can be used to recover the known expression of free and Boolean cumulants in terms of monotone ones \cite{arizmendi_15}. Notice that the proof is different from the one we used in \cite{ebrahimipatras_18}. In Section \ref{sec:Fla2}, we use the same strategy to handle the more involved computation of multivariate monotone cumulants in terms of free (or Boolean) cumulants. It is based on a recursion that has a similar structure than the one defining the coefficients of the computation of the continuous Baker--Campbell--Hausdorff (BCH) coefficients $\omega$ in a Hall basis described in Murua's work \cite{murua_07}. This should not come as a surprise since the BCH coefficients are known to be closely related to the Magnus formula. However, their appearance in the context of cumulant-cumulant relations in non-commutative probability theory seems to indicate new perspectives. In the process of adapting Murua's constructions to the context of non-crossing partitions, we introduce the notion of quasi-monotone partitions. It allows to define a statistics on non-crossing partitions that leads to the computation of the coefficients of the multivariate monotone-free cumulant-cumulant expansion.

\medskip

In the following the pair $(A,\varphi)$ denotes a non-commutative probability space. Its unital linear map $\varphi\colon A \to \mathbb{K}$ sends elements from an arbitrary unital associative algebra of random variables into the ground field of characteristic zero. We remark that in the context of the present paper any additional assumptions such as positivity on the map $\varphi$ are not necessary. We assume some familiarity with the basic combinatorial notions and constructions related to free probability and non-crossing partitions. On these topics, the reader is referred to \cite{nicaspeicher_06}.

Iterated left- and right-multiplication operators are denoted as follows. Given any bilinear product, denoted by $\bullet$, we define the left/right multiplication operators $L^{m}_{\alpha\scriptscriptstyle{\bullet} }(\beta):=L^{m-1}_{\alpha\scriptscriptstyle{\bullet} }(\alpha {\bullet} \beta)$, $L^{0}_{\alpha\scriptscriptstyle{\bullet} }(\beta)=\beta$ respectively $R^{m}_{\scriptscriptstyle{\bullet} \alpha}(\beta):=R^{m-1}_{\scriptscriptstyle{\bullet} \alpha}(\beta {\bullet} \alpha)$, $R^{0}_{\scriptscriptstyle{\bullet} \alpha}(\beta)=\beta$. 

Recall at last the classical generating series for the Bernoulli numbers $B_n$ 
$$
	\frac{z}{\exp(z)-1} = 1 + \sum_{n>0} \frac{B_n}{n!} z^{n}
$$	
as well as its inverse
$$
	\frac{\exp(z)-1}{z} = 1 + \sum_{n>0} \frac{1}{(n+1)!} z^{n}.
$$

\vspace{0.5cm}


{\bf{Acknowledgements}}: The fourth author would like to thank CONACyT (Mexico) for its support via the scholarship 714236.  This work was partially supported by the project ``Pure Mathematics in Norway'', funded by Trond Mohn Foundation and Troms{\o} Research Foundation. 


\section{Non-crossing partitions and rooted trees}
\label{sec:ncpart}

We denote by $\mathrm{NC}(n)$ the lattice of non-crossing partitions of the set $[n]=\{1,\ldots,n\}$. Sometimes the obvious extension of the notion of non-crossing partition of $[n]$ to other finite totally ordered sets $X$ is used. The corresponding set of non-crossing partitions will then be written $\mathrm{NC}(X)$. 

The number of blocks of a non-crossing partition, $\pi = \{\pi_1,\ldots,\pi_m\} \in \mathrm{NC}(n)$, is denoted by $|\pi|=m$. A non-crossing partition in $\mathrm{NC}(n)$ is irreducible if and only if both $1$ and $n$ are in the same block. The set of irreducible non-crossing partitions in $\mathrm{NC}(n)$ is denoted by $\mathrm{NC}^{\scriptscriptstyle\mathrm{irr}}(n)$. The sets of non-crossing and irreducible non-crossing partitions with $k$ blocks are denoted $\mathrm{NC}_k(n)$ respectively $\mathrm{NC}_k^{\scriptscriptstyle\mathrm{irr}}(n)$. Interval partitions are a subset of $\mathrm{NC}(n)$ and form the lattice $\mathrm{I}(n)$. 

There is a natural partial order on the lattice $\mathrm{NC}(n)$ called reversed refinement order $\leq$. For $\pi,\sigma\in \mathrm{NC}(n)$, we write ``$\pi \leq \sigma$'' if every block of $\sigma$ is a union of blocks of $\pi$. The maximal element of $\mathrm{NC}(n)$ with this order is $1_n:=\{\{1,\dots,n\}\}$ (the partition of $[n]$ with only one block), and the minimal element is $0_n:=\{\{1\},\{2\},\dots,\{n\}\}$ (the partition of $[n]$ with $n$ blocks). This order is at play when we refer to the lattice structure on $\mathrm{NC}(n)$. 

\begin{rem}\label{rem:minmax}
For $\pi,\sigma \in \mathrm{NC}(n)$, we write ``$\pi \ll \sigma$'' for the so-called {min-max order} on $\mathrm{NC}(n)$, meaning that $\pi \leq \sigma$ and that for every block $\sigma_i$ of $\sigma$ there exists a block $\pi_j$ of $\pi$ such that $\min(\sigma_i),\max(\sigma_i) \in \pi_j$. We refer the reader to \cite{BN_08} for a broader description of this partial order and its deeper meaning in the context of moments and cumulants. 
\end{rem}

Recall that a rooted tree $t$ is a connected and simply connected graph with a distinguished vertex called the root. The edges are oriented towards the root. Vertices have exactly one outgoing edge and an arbitrary number of incoming ones, except for the root which is the only vertex with no outgoing edge. The leaves are the vertices without incoming edges. A forest is a finite set of rooted trees.

Let $\mathcal{T}$ denote the set of non-planar rooted trees 
\[
	\mathcal{T} = \left\{\begin{array}{c}
		\Forest{[]},
		\Forest{[[]]},
		\Forest{[[[]]]},
		\Forest{[[][]]},
		\Forest{[[[[]]]]},
		\Forest{[[][][]]},
		\Forest{[[[]][]]},
		\Forest{[[[][]]]},
		\ldots
 		\end{array}
		\right\}.
\]

For later use we denote by $\mathcal{T}^\ell$ the set of so-called ladder trees 
$$
	\ell_1= \scalebox{0.8}{\Forest{[]}}
	\quad
	\ell_2= \scalebox{0.8}{\Forest{[[]]}}
	\quad
	\ell_3= \scalebox{0.8}{\Forest{[[[]]]}}
	\quad
	\ell_4= \scalebox{0.8}{\Forest{[[[[]]]]}}
	\quad
	\ell_5= \scalebox{0.8}{\Forest{[[[[[]]]]]}}
	\quad
	\cdots.
$$

The degree, $|t|$, of a rooted tree $t \in \mathcal{T}$ is defined by its number of vertices and $\mathcal{T}_n$ contains all trees of degree $n$. The empty tree, $\emptyset$, has degree $|\emptyset|=0$. 
Recall also that if a rooted tree $t$ is obtained by attaching trees $t_1,\dots,t_n$ on a new common root, the tree factorial of $t$ is such that $t! := |t|t_1 ! \dots t_n !$. Together with $\bullet !:=1$, this identity defines inductively tree factorials.

\begin{lemma}[\cite{kreimer_00}] \label{lem:kreimer}
Let $t \in \mathcal{T}$ be a rooted tree. By $\mathcal{L}_-(t) \subset \mathcal{T}$ we denote the multiset of rooted trees that result from eliminating a leaf together with its outgoing edge from $t$. The cardinality of $\mathcal{L}_-(t)$ equals the number of leaves of $t$. Then
\allowdisplaybreaks
\begin{equation}
\label{Kreimer}
	\frac{|t|}{t!} = \sum_{t' \in \mathcal{L}_-(t)} \frac{1}{t'!}.
\end{equation}  
\end{lemma}
$$
	\frac{2}{\scalebox{0.4}{\Forest{[[]]}}!} 
	=\frac{|\scalebox{0.4}{\Forest{[[][]]}}|}{\scalebox{0.4}{\Forest{[[][]]}}!},
	\quad
	\frac{1}{\scalebox{0.4}{\Forest{[[][]]}}!} + \frac{1}{\scalebox{0.4}{\Forest{[[[]]]}}!}
	= \frac{1}{3} + \frac{1}{3!} 
	=\frac{1}{2} 
	= \frac{|\scalebox{0.4}{\Forest{[[[]][]]}}|}{\scalebox{0.4}{\Forest{[[[]][]]}}!},
	\quad
	\frac{2}{\scalebox{0.4}{\Forest{[[[]]]}}!}
	= \frac{1}{3} 
	= \frac{|\scalebox{0.4}{\Forest{[[[][]]]}}|}{\scalebox{0.4}{\Forest{[[[][]]]}}!}, 
	\quad
	\frac{2}{\scalebox{0.4}{\Forest{[[[]][]]}}!}
	+\frac{1}{\scalebox{0.4}{\Forest{[[][][]]}}!}
	= \frac{1}{4} + \frac{1}{4} 
	= \frac{|\scalebox{0.4}{\Forest{[[[]][][]]}}|}{\scalebox{0.4}{\Forest{[[[]][][]]}}!}, 
$$ 
 
Recall that the blocks of a non-crossing partition $\pi=\{\pi_1,\dots,\pi_k\}$ are naturally ordered: $\pi_j<\pi_i$ if and only if there exist elements $x,y$ in $\pi_j$ such that $x<z<y$ for any $z$ in $\pi_i$.
A block $\pi_i$  is outer if it is minimal for this order, that is if there do not exist elements $x,y$ in another block $\pi_j$ such that $x<z<y$ for any $z$ in $\pi_i$.

To any irreducible non-crossing partition ${\pi=\{\pi_1,\ldots,\pi_k\} \in \mathrm{NC}^{\scriptscriptstyle\mathrm{irr}}(n)}$ one can naturally associate a rooted tree $t(\pi)$ encoding the hierarchy of the nested blocks of $\pi$. We refer to  \cite{arizmendi_15} for more details. The root is associated with the outer block, that contains $1,n \in [n]$. 

To a non-crossing partition which is not irreducible, one associates similarly a forest of trees, each tree encoding the hierarchy of one of its irreducible components. To define the latter, consider the block $B_1$ of $\pi$ that contains $1=:i_{1,min}$ and write $i_{1,max}$ its maximal element. The set of blocks $\rho_1:=\{\pi_i|\pi_i\geq B_1\}$ forms an irreducible non-crossing partition of the set $\{i_{1,\min},i_{1,\min}+1,\dots,i_{1,\max}\}$. Consider then the block $B_2$ with minimal element $i_{2,\min}:=i_{1,\max}+1$ and write $i_{2,\max}$ for its maximal element. The set of blocks $\rho_2:=\{\pi_i,\pi_i\geq B_2\}$ forms an irreducible non-crossing partition of $\{i_{2,\min},\dots,i_{2,\max}\}$. Iterating this construction, $\pi$ decomposes uniquely, $\pi=\rho_1\cup \rho_2\cup \cdots \cup \rho_l$, as a union of irreducible non-crossing partitions. The partitions $\rho_1,\dots,\rho_l$ are the irreducible components of $\pi$.
The latter are in bijection with the outer blocks which, in turn, are associated to the roots of the trees in the forest. Note that for $\pi \in \mathrm{NC}(n)$, $|\pi| = |t(\pi)|$. For instance, consider the following two non-crossing partitions and the corresponding rooted tree, respectively forest, encoding the nesting structures 
$$
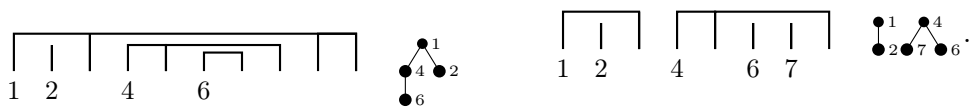
\begin{array}{cc}
\begin{tikzpicture}[thick,font=\small]
     \path 	(0,0) 		node (a) {1}
           	(0.5,0) 	node (b) {2}
           	(1,0) 		node (c) {\phantom{1}}
           	(1.5,0) 	node (d) {4}
           	(2,0) 		node (e) {\phantom{1}}
           	(2.5,0) 	node (f)  {6}
           	(3,0) 		node (g) {\phantom{1}}
           	(3.5,0) 	node (h) {\phantom{1}}
		(4,0) 		node (i) {\phantom{1}}
		(4.5,0) 	node (j) {\phantom{1}};
     \draw (a) -- +(0,0.75) -| (j);
     \draw (b) -- +(0,0.60) -| (b);
     \draw (c) -- +(0,0.75) -| (c);
     \draw (d) -- +(0,0.60) -| (e);
     \draw (e) -- +(0,0.60) -| (h);
     \draw (f) -- +(0,0.50) -| (g);
     \draw (i) -- +(0,0.75) -| (j);
   \end{tikzpicture} 	& \\[-1.1cm]
				&    \Forest{[1[2][4[6]]]}
\end{array}    
\qquad
\begin{array}{cc}
\begin{tikzpicture}[thick,font=\small]
     \path 	(0,0) 		node (a) {1}
           	(0.5,0) 	node (b) {2}
           	(1,0) 		node (c) {\phantom{1}}
           	(1.5,0) 	node (d) {4}
           	(2,0) 		node (e) {\phantom{1}}
           	(2.5,0) 	node (f)  {6}
           	(3,0) 		node (g) {7}
           	(3.5,0) 	node (h) {\phantom{1}};
     \draw (a) -- +(0,0.75) -| (c);
     \draw (b) -- +(0,0.60) -| (b);
     \draw (c) -- +(0,0.75) -| (c);
     \draw (d) -- +(0,0.75) -| (e);
     \draw (e) -- +(0,0.75) -| (h);
     \draw (f)  -- +(0,0.60) -|(f);
     \draw (g) -- +(0,0.60)-|(g); 
   \end{tikzpicture} 	& \\[-1.1cm]
				&    \Forest{[1[2]]}\Forest{[4[6][7]]}.\\
				&\\
				&
\end{array}  
$$
Here, for notational clarity, the nodes of the trees are decorated by the minimal elements of the corresponding blocks in the partition.

Monotone partitions are non-crossing partitions equipped with a total order on their blocks refining the natural partial order of the blocks. We refer the reader to \cite{arizmendi_15} for more details on monotone partitions. Choosing such a total order amounts to reindexing the blocks in such a way that $i<j$ implies $\pi_i \leq \pi _j$. Let us write $m(\pi)$ for the number of these total orders, that is, the number of monotone partitions associated to a given non-crossing partition, then
\begin{equation}
\label{monolab}
	m(\pi)=\frac{|t(\pi)|!}{t(\pi)!},
\end{equation}
where we recall that $|t(\pi)|=|\pi|$ stands for the number of blocks of $\pi$. Following \cite{hasebesaigo_11} we introduce a random process that generates the set denoted $\mathcal{M}^{{\mathrm{irr}}}_k(n)$ of all irreducible monotone partitions (monotone partitions of irreducible non-crossing partitions) of $[n]$ with $k$ blocks. In fact, this construction is adapted from standard arguments on the construction and enumeration of non-crossing partitions. See, e.g., \cite{nicaspeicher_06}. We call interval in a finite subset $S$ of the set of integers any sequence of successive elements of $S$ for the usual order. We also call interval the associated subset of $S$. For example, $\{4,8\}$ is an interval in the set  $\{2,3,4,8,10\}$. 

\begin{itemize}
\item To initiate the recursion, choose an interval $\pi_k$ in $[n]$ of length less or equal to $n-k$ such that $1,n\notin \pi_k$. Call $S_2$ the complementary subset of $\pi_k$ in $[n]$ and denote $n_2$  its cardinality (so that $n_2\geq k$). 

\item For $i=2,\ldots, k-1$, choose an interval $\pi_{k-i+1}$ in $S_i$ of length less that $n_i-k+i-1$ such that $1,n\notin \pi_{k-i+1}$. Call $S_{i+1}$ the complementary subset of $\pi_{k-i+1}$ in $S_i$ and denote $n_{i+1}$ its cardinality (so that $n_{i+1}\geq k-i+1$). 

\item Set $\pi_1:=S_k$ (so that $1,n\in \pi_1$).
\end{itemize}

Then the ordered sequence $\pi_1,\dots,\pi_k$ is an irreducible monotone partition of $[n]$. The process creates (randomly) all irreducible monotone partitions of $[n]$ with $k$ blocks. As a corollary of this construction it follows that non trivial irreducible monotone partitions of $[n]$ are in bijection with pairs $(\alpha,\kappa)$ where $\alpha$ is an interval of $[n]$ that does not contain $\{1,n\}$ (as $\pi_k$ above) and $\kappa$ runs over all irreducible monotone partitions of $[n]\backslash\alpha$. This corollary is all we will need and follows easily from the definitions, but we find the description of the generating process of irreducible monotone partitions illuminating in view of our forthcoming developments.

\section{Pre-Lie algebra and cumulants}
\label{sec:preLie}

In this section we outline the mathematical setting in which we will express monotone cumulants in terms of free and Boolean cumulants and vice versa. 

\smallskip

Let $(A,\varphi)$ be a non-commutative probability space \cite{nicaspeicher_06}. The starting point is the non-unital tensor algebra $T_+(A) := \bigoplus_{n>0} A^{\otimes n}$. Elements $w \in T_+(A)$ are denoted as words, i.e., $w=a_1 \cdots a_k := a_1 \otimes \cdots \otimes a_k \in A^{\otimes k}$. The length of a word is defined by its number of letters and will be denoted $\mathrm{deg}(w)=k$. If $\pi=\{i_1,\dots,i_m\}$ is a subset of $[k]=\{1,\dots,k\}$, $a_\pi$ stands for the word $a_{i_1}\cdots a_{i_m} \in T_+(A)$. The following notation is put in place for linear forms $\alpha \in \mathfrak{g}:=T_+(A)^\star$\begin{equation}
    \label{trickynotation}
	\alpha_{\pi}(w):=\prod_{\pi_i \in \pi}\alpha(w_{\pi_i}), 
\end{equation}
where $\pi = \{\pi_1,\ldots,\pi_k\} \in \mathrm{NC}(n)$. The functional $\varphi: A \to \mathbb{K}$ is extended linearly to $T_+(A)$ giving the $n$th multivariate moment 
$$
	\phi(w):=\varphi(a_1 \cdot_{\!\!\scriptscriptstyle{A}} \cdots \cdot_{\!\!\scriptscriptstyle{A}} a_n)
$$
for $w=a_1\cdots a_n\in T_+(A)$, where $\cdot_{\!\!\scriptscriptstyle{A}}$ denotes the product in $A$ that we distinguish notationally from the concatenation product of words in the tensor algebra. 

Free, Boolean and monotone cumulants form respectively the families of multilinear functionals $\{r_n \colon A^{\otimes n}\to \mathbb{K}\}_{n\geq1}$, $\{ b_n :A^{\otimes n}\to \mathbb{K}\}_{n \ge 1}$, $\{h_n \colon A^{\otimes n}\to \mathbb{K} \}_{n \ge 1}$. These families can be viewed equivalently as linear forms $\kappa,\beta,\rho: T_+(A)\to \mathbb{K}$, i.e., $\kappa(a_1\cdots a_n):=r_n(a_1,\ldots, a_n)$ and analogously for the Boolean and monotone cumulants. Free \cite{nicaspeicher_06}, Boolean \cite{speicher_97c} and monotone \cite{hasebesaigo_11} cumulants are defined in terms of the corresponding moment-cumulant relations
\begin{align}
	\phi(a_1 \cdots a_n)
	&= \sum_{\pi\in \NC(n)} \frac{1}{t(\pi)!} 
	\prod_{\pi_i \in \pi} \rho(a_{\pi_i}) \label{mono}\\
	&= \sum_{\pi\in \NC(n)} \prod_{\pi_i \in \pi}\kappa(a_{\pi_i})\label{free}\\
	&= \sum_{\pi\in \mathrm{I}(n)} \prod_{\pi_i \in \pi}\beta(a_{\pi_i})\label{boolean}.
\end{align}

The following proposition is central to the approach used in this paper. It defines a (left) pre-Lie algebra structure on cumulants, seen as linear forms over $T_+(A)$. 
 
\begin{prop}
\label{prop:bridges}
Let $\alpha,\beta \in \mathfrak{g}$ and $w \in T_+(A)$, then the product $\rhd \colon \mathfrak{g} \otimes \mathfrak{g} \to \mathfrak{g}$ defined by
\begin{equation}
	\alpha \rhd \beta (w)
	:= - \sum_{w_1w_2w_3 = w \atop \mathrm{deg}(w_i)>0}
	 \beta(w_1w_3)\alpha(w_2) 	\label{bridge1}
\end{equation}	
satisfies the left pre-Lie identity
\begin{equation}
\label{preLierelation}
	\alpha \rhd (\beta \rhd \gamma) - (\alpha \rhd \beta) \rhd \gamma
	= \beta \rhd (\alpha \rhd \gamma) - (\beta \rhd \alpha) \rhd \gamma.
\end{equation}	
\end{prop}

\begin{proof} 
Let $w$ be a word in $T_+(A)$. In the following formulas all words and subwords appearing in summations are non-empty except possibly $w_{13}$, $w_{31}$ and $w_3'$. We get:
\begin{align*}
	\alpha \rhd (\beta \rhd \gamma)(w) 
	&=\sum_{w_1w_2w_3=w}\Big( \sum_{w_{11}w_{12}w_{13}=w_1}
	 				\gamma(w_{11}w_{13}w_3)\beta(w_{12})\alpha(w_2)\\
	&+\sum_{w_{11}w_{12}=w_1 \atop w_{32}w_{33}=w_3}
					\gamma(w_{11}w_{33})\beta(w_{12}w_{32})\alpha(w_2)
	+\sum_{w_{31}w_{32}w_{33}=w_3}
					\gamma(w_1w_{31}w_{33})\beta(w_{32})\alpha(w_2)\Big),
\end{align*}
and
$$
	(\alpha \rhd \beta) \rhd \gamma(w)
	=\sum_{w_{11}w_2'w_{33}=w}\
	   \sum_{w_{12}w_2w_{32}=w_2'}\gamma(w_{11}w_{33})\beta(w_{12}w_{32})\alpha(w_2),
$$
so that
$$
	(\alpha \rhd (\beta \rhd \gamma) - (\alpha \rhd \beta) \rhd \gamma)(w)
	=\sum_{w_1w_2w_3'w_4w_5=w}
	\gamma(w_1w_3'w_5)\big(\beta(w_2)\alpha(w_4)+\beta(w_4)\alpha(w_2)\big).
$$
As the expression is symmetric in $\alpha$ and $\beta$, the statement follows.
\end{proof}

\begin{rem}
Recall that pre-Lie algebras are Lie admissible \cite{AG_81,burde_06,manchon_11}, that is, the pre-Lie algebra $\mathfrak{g}$ is a Lie algebra for the commutator bracket $[f,g]:=f \rhd g - g \rhd f.$ The notion of pre-Lie algebra appeared in different place in mathematics, including algebra and geometry, but also in control theory, where preLie algebras are known as chronological algebras \cite{AG_81,burde_06,cartier_11,manchon_11}. 
\end{rem}

The product \eqref{bridge1} may be formulated using irreducible non-crossing partitions. Indeed, we see immediately that
\begin{align}
	\alpha \rhd \beta (a_1\cdots a_n)
		& = - \sum_{\pi \in \mathrm{NC}^{\scriptscriptstyle\mathrm{irr}}_2(n)}
		  	  	\beta(a_{\pi_1})\alpha(a_{\pi_2}). \label{bridge3}
\end{align}	
Using the standard pictorial representation, elements in $\pi=\{\pi_1,\pi_2\} \in \mathrm{NC}^{\scriptscriptstyle\mathrm{irr}}_2(n)$ have the particular form
$$
\begin{tikzpicture}[thick,font=\small]
     \path 	(0,0) 		node (a) {1}
           	(0.5,0.4) 	node (b) {}
           	(1,0) 		node (c) {}
           	(1.5,0) 	node (d) {}
           	(2,0.3) 	node (e) {}
           	(2.5,0) 	node (f)  {}
           	(3,0) 		node (g) {}
           	(3.5,0.4) 	node (h) {}
		(4,0) 		node (i) {$n$};	
     \draw (a) -- +(0,0.75) -| (i);
     \draw (b) node {.\ .\ .} ;
     \draw (c) -- +(0,0.75) -| (c);
     \draw (d) -- +(0,0.60) -| (f);
     \draw(e) node {.\ .\ .}  ;
     \draw (g) -- +(0,0.75) -| (g);
     \draw (h) node {.\ .\ .}  ;
   \end{tikzpicture} 
$$ 
where the outer block $\pi_1$ contains $1$ and $n$.

In the univariate case the explicit computations of pre-Lie products in $\mathfrak{g}$ give:
\begin{align*}
	\alpha \rhd \beta(a^2) &= 0\\
	\alpha \rhd \beta(a^3) 
	&= - \beta(a^2)\alpha(a)\\
		\alpha \rhd \beta(a^4) 
	&= - 2 \beta(a^3)\alpha(a) - \beta(a^2)\alpha(a^2)\\
	\alpha \rhd \beta(a^5) 
	&= - 3 \beta(a^4)\alpha(a) - 2 \beta(a^3)\alpha(a^2) 
	- \beta(a^2)\alpha(a^3)\\
	\alpha \rhd \beta(a^6) 
	&= - 4 \beta(a^5)\alpha(a) - 3 \beta(a^4)\alpha(a^2) - 
	2 \beta(a^3)\alpha(a^3) - \beta(a^2)\alpha(a^4).
\end{align*}
This can be summarised in a closed formula based on the fact that there are $n-l-1$ intervals of length $l$ contained in $\{2,\dots,n-1\}$:

\begin{cor}
\label{cor:sticks}
Let $\alpha,\beta \in \mathfrak{g}$ and $a \in A$. Then:
\begin{equation}
\label{sticks1}
	\alpha \rhd \beta(a^n) 
	= - \sum_{l=1}^{n} (n-l-1) \beta(a^{n-l} )\alpha(a^l).
\end{equation}
\end{cor}

Proposition \ref{prop:bridges} points at an interesting phenomenon. For $\alpha_1,\ldots,\alpha_4 \in \mathfrak{g}$, we note that in general $\alpha_1 \rhd \alpha_2 (w)=0$ if $\mathrm{deg}(w)<3$ and $\alpha_1 \rhd (\alpha_2 \rhd \alpha_3) (w)=0$ if $\mathrm{deg}(w)<4$, while $(\alpha_1 \rhd \alpha_2) \rhd \alpha_3 (w)=0$ if $\mathrm{deg}(w)<5$ and $(\alpha_1 \rhd \alpha_2) \rhd (\alpha_3 \rhd \alpha_4) (w)=0$ if $\mathrm{deg}(w)<6$. We may therefore associate an {\emph{effective degree}} to any pre-Lie monomial in $\mathfrak{g}$. It determines the minimal length of words in the support of such monomials. The effective degree for any pre-Lie monomial $P$ of cumulants is denoted $\#(P)$. 

\begin{lemma}\label{effect}
Let $P=P_1 \rhd P_2 $ be a pre-Lie monomial in $\mathfrak{g}$, $P_1,P_2 \in \mathfrak{g}$. Its effective degree computes recursively
\begin{equation}
\label{effdegree}
	\#(P) = \#(P_1) + \max(2,\#(P_2)).
\end{equation}   
\end{lemma}

\vspace{0.3cm}

We introduce now the fundamental tool for our forthcoming computation of a formula expressing monotone cumulants in terms of free and Boolean cumulants: the (pre-Lie) Magnus expansion. The classical Magnus expansion \cite{magnus_54} is a well-known object in applied mathematics \cite{blanes_09,iserles_00}. In the context of ordinary differential equations, it computes the logarithm of the solution of a linear matrix-valued initial value problem. Its pre-Lie content was first uncovered in \cite{AG_81}. It was then rediscovered and further developed in \cite{ebrahimimanchon_09Mag,ebrahimimanchon_09}. The corresponding endomorphism $\Omega'$ of a pre-Lie algebra satisfying suitable completion properties appeared in the context of enveloping algebras of pre-Lie algebras in \cite{AG_81}, together with its inverse, denoted $W$. The Hopf and group-theoretical properties of $\Omega'$ and its links with the Mielnik--Plebanskii--Strichartz continuous Baker--Campbell--Hausdorff formula have been investigated in \cite{chapat_13,ebrahimipatras_14}. We refer the reader to \cite{Nacer,iserles_02} for more details, including its range of applications and mathematical properties. 

\begin{defi}\cite{chapat_13,ebrahimimanchon_09Mag} (pre-Lie Magnus expansion)\label{def:preLieMagnus}
The pre-Lie Magnus expansion $\Omega' \colon \mathfrak{g} \to \mathfrak{g}$ is defined by
\begin{equation}
\label{preLieMag1}
	\Omega'(\kappa)
	:=\frac{L_{\Omega'(\kappa) \scriptscriptstyle \rhd}}
	{e^{L_{\Omega'(\kappa) \scriptscriptstyle \rhd}} - 1} (\kappa).
\end{equation}
\end{defi}
The first few terms in the expansion \eqref{preLieMag1} are given
\begin{align}
\begin{split}
\label{preLieMag2}
\Omega'(\kappa)
	&= \sum_{n \ge 0}\frac{B_n}{n!} L^{n}_{\Omega'(\kappa) \scriptscriptstyle \rhd}(\kappa)\\
	&=\kappa -\frac{1}{2} \kappa \rhd \kappa + 
	\frac{1}{4} (\kappa \rhd \kappa)\rhd \kappa + \frac{1}{12} \kappa \rhd (\kappa \rhd \kappa)	
	- \frac{1}{8} \big( ( \kappa \rhd \kappa) \rhd \kappa \big) \rhd \kappa \\ 	
	&  
		- \frac{1}{24} \kappa \rhd \big((\kappa \rhd \kappa) \rhd \kappa \big) 
		- \frac{1}{24} (\kappa \rhd \kappa) \rhd (\kappa \rhd  \kappa)
		- \frac{1}{24} \big( \kappa \rhd (\kappa \rhd \kappa)\big) \rhd \kappa   \\
	&
	     	- \frac{1}{720} \kappa \rhd ( \kappa \rhd ( \kappa \rhd ( \kappa \rhd \kappa)))   
		 	+ \cdots  .
\end{split}
\end{align}
The compositional inverse of $\Omega'$ is $W\colon \mathfrak{g} \to \mathfrak{g}$
\begin{align}
	W(\kappa) 
	&:= \frac{e^{L_{\kappa \scriptscriptstyle \rhd}} - 1} {L_{\kappa \scriptscriptstyle \rhd}}(\kappa) 		\label{preLieMaginv1}\\ 
	&=\kappa + \sum_{n>0} \frac{1}{(n+1)!}L^{n}_{\kappa \scriptscriptstyle \rhd}(\kappa)
	= \kappa + \frac{1}{2!} \kappa \rhd \kappa + \frac{1}{3!} \kappa \rhd ( \kappa \rhd \kappa) + \cdots. 	\label{preLieMaginv2}
\end{align}

\begin{rem}
By identifying the binary product $(M \rhd N)(t):= \int_0^t ds \int_0^s du [\dot{M}(s),\dot{N}(u)]$ as a pre-Lie product on time-dependent operators -- seen as elements in a non-unital algebra $A$ of operators, having suitable regularity properties allowing to compute derivatives, integrals, and so on, -- one recovers the classical Magnus expansion \cite{AG_81,ebrahimimanchon_09Mag}. See \cite{Nacer} for a detailed review.
\end{rem}

The pre-Lie Magnus expansion and its inverse permit to express monotone cumulants in terms of Boolean and free cumulants and vice versa. 

\begin{theo}\label{thm:MonFreeBool}
Monotone, free and Boolean cumulants, $\rho, \kappa, \beta \in \mathfrak{g}$ are related in terms of the pre-Lie Magnus expansion \eqref{preLieMag1} 
\begin{equation}
\label{MagnusKey1}
	\rho 	= \Omega'(\kappa) 
		= -\Omega'(-\beta).
\end{equation}
\end{theo}

\begin{proof}
The proof of \eqref{MagnusKey1} follows from a group-theoretical result in the context of Hopf algebra and requires shuffle algebra arguments. For details we refer the reader to \cite{ebrahimipatras_15,ebrahimipatras_16,ebrahimipatras_18,ebrahimipatras_19,ebrahimipatras_20}. 
\end{proof}

The above theorem implies that 
\begin{equation}
\label{MagnusKey2}
		\kappa 
		= W(-\Omega'(-\beta))
		\qquad
		\beta  
		= -W(-\Omega'(\kappa)),
\end{equation}
which can be expressed in more compact form
\begin{equation}
\label{MagnusKey3}
	\kappa 
	=e^{-L_{\Omega'(-\beta) \scriptscriptstyle \rhd}}(\beta)
	\qquad
	\beta
	=e^{-L_{\Omega'(\kappa) \scriptscriptstyle \rhd}}(\kappa).
\end{equation}
In \cite{ebrahimipatras_18} it was shown that \eqref{MagnusKey3} permits to describe the relations between free and Boolean cumulants in terms of closed formulas defined via irreducible non-crossing partitions. They were first derived in \cite{BN_08} using M\"obius calculus. See also \cite{arizmendi_15}.

\begin{prop}\label{prop:freebooleanrel}
\begin{eqnarray*}
	{\beta}(a_1\cdots a_n) 
		&=&  \sum\limits_{\pi\in \NC^{\scriptscriptstyle\mathrm{irr}}(n)}
			\prod_{\pi_i \in \pi}\kappa(a_{\pi_i}),\\ 
	{\kappa}(a_1\cdots a_n) 
		&=&  \sum\limits_{\pi\in \NC^{\scriptscriptstyle\mathrm{irr}}(n)} 
			(-1)^{|\pi|-1}\prod_{\pi_i \in \pi}{\beta}(a_{\pi_i}).
\end{eqnarray*}
\end{prop}

In the forthcoming sections of the article, we will first invert formulas \eqref{MagnusKey1} to describe (known) results of expressing multivariate free and Boolean cumulants in terms of monotone cumulants using irreducible non-crossing partitions. After that we will address the inverse problem of expressing multivariate monotone cumulants in terms of free and Boolean cumulants using irreducible non-crossing partitions.

\section{Iterated pre-Lie products and rooted trees}
\label{sec:Fla1}

In this section we compute iterated pre-Lie products in $\mathfrak{g}$ and apply the results to describe cumulant-cumulant relations. 

\smallskip

Observe that for a word $w=a_1 \cdots a_m$ of length $\deg(w)=m>4$ and $\alpha,\beta,\kappa \in \mathfrak{g}$, we have
\allowdisplaybreaks
\begin{align}
	(\alpha \rhd \kappa) \rhd \beta (w) 
	&= \sum_{w_1w_2w_3=w \atop \deg(w_i)>0} 
	-\beta(w_1w_3) (\alpha \rhd \kappa)(w_2)\\
	&= \sum_{w_1w_2w_3=w \atop \deg(w_i)>0} 
	   \sum_{w_{2\scriptscriptstyle{1}}w_{{22}}w_{2\scriptscriptstyle{3}}=w_2 \atop \deg(w_{2i})>0} 
	\beta(w_1w_3) \kappa(w_{21}w_{23}) \alpha(w_{22})\\
	&= \sum_{\pi \in \mathrm{NC}^{\scriptscriptstyle\mathrm{irr}}_3(m) \atop {\pi_1\leq\pi_2\leq\pi_3}} 
	\beta(a_{\pi_1})\kappa(a_{\pi_2})\alpha(a_{\pi_3}).
\end{align}
For $\alpha=\beta=\kappa$ this simplifies to
$$
	R^2_{\scriptscriptstyle{\rhd}\kappa }(\kappa)(w)
	=(\kappa \rhd \kappa) \rhd \kappa (w) 
	= \sum_{\pi \in \mathrm{NC}^{\scriptscriptstyle\mathrm{irr}}_3(m) \atop t(\pi) =\ell_3} 
	\prod\limits_{i=1}^3\kappa(a_{\pi_i}),
$$
where we expressed the constraint $\pi_1\leq\pi_2\leq\pi_3$ on the blocks of $\pi \in \mathrm{NC}^{\scriptscriptstyle\mathrm{irr}}_3(m)$ in terms of the corresponding hierarchy tree. More generally, we have that $\pi=\{\pi_1, \ldots,\pi_{n+1}\} \in \mathrm{NC}^{\scriptscriptstyle\mathrm{irr}}_{n+1}(m)$ with $\pi_1\leq \dots\leq\pi_{n+1}$ corresponds to  $t(\pi) =\ell_{n+1}$. The next proposition covers the general case.

\begin{prop}\label{prop:rightiteration}
Let $\kappa_1,\dots,\kappa_{n+1}\in \mathfrak{g}$  and $w =a_1\cdots a_m\in T_+(A)$ a word of length $\deg(w) =m \ge 2n+1$, $n>0$. Then 
\begin{equation}
	(\cdots((\kappa_1\rhd\kappa_2)\rhd \kappa_3)\cdots \kappa_n)\rhd \kappa_{n+1}(w)
	=  \sum_{\pi \in \mathrm{NC}^{\scriptscriptstyle\mathrm{irr}}_{n+1}(m) \atop {t(\pi) =\ell_{n+1}}} 
	(-1)^{n}\kappa_{n+1}(a_{\pi_1})\cdots \kappa_{1}(a_{\pi_{n+1}}).
	\label{obs1}
\end{equation}
In the univariate case this simplifies to
\begin{equation}
	R^{n}_{\scriptscriptstyle{\rhd}\kappa }(\kappa)(w)
	=  \sum_{\pi \in \mathrm{NC}^{\operatorname{irr}}_{n+1}(m) \atop t(\pi) =\ell_{n+1}} 
	(-1)^{n}\prod\limits_{i=1}^{n+1}\kappa(a_{\pi_i}).
	\label{obs1bis}
\end{equation}
\end{prop}

\begin{proof}
For notational simplicity we give the proof in the univariate case, the general case follows by the same computation.
By induction we see that 
\allowdisplaybreaks
\begin{align*}
	R^{n+1}_{\scriptscriptstyle{\rhd}\kappa }(\kappa)(w)
	&= R^{n}_{\kappa \scriptscriptstyle \rhd} (\kappa) \rhd\kappa (w)\\
	&= \sum_{w_1w_2w_3=w \atop \deg(w_i)>0} 
	-\kappa(w_1w_3) R^{n}_{\kappa \scriptscriptstyle \rhd} (\kappa)(w_2)\\
	&= \sum_{w_1w_2w_3=w \atop \deg(w_i)>0} 
	-\kappa(w_1w_3) 
	\sum_{\pi' \in \mathrm{NC}^{\scriptscriptstyle\mathrm{irr}}_{n+1}(\deg(w_2)) \atop t(\pi') =\ell_{n+1}} 
	(-1)^{n}\prod\limits_{i=1}^{n+1}\kappa(a_{\pi'_i}).\\
	&=\sum_{\pi \in \mathrm{NC}^{\scriptscriptstyle\mathrm{irr}}_{n+2}(m) \atop t(\pi) =\ell_{n+2}} 
	(-1)^{n+1}\prod\limits_{i=1}^{n+2}\kappa(a_{\pi_i}).
\end{align*}
\end{proof}

Let us now turn to left iteration of the pre-Lie product on cumulants. In degree three, we have 
\allowdisplaybreaks
\begin{align}
	\alpha \rhd (\kappa \rhd \beta) (w) 
	&= \sum_{w_1w_2w_3=w \atop \deg(w_i)>0} 
	-(\kappa \rhd \beta)(w_1w_3)\alpha(w_2) \\
	&= \sum_{w_1w_2w_3=w \atop \deg(w_i)>0} 
	   \sum_{w_{11}w_{12}w_{13}=w_1 \atop \deg(w_{11})>0,\deg(w_{12})>0} 
	\beta(w_{11}w_{13}w_3) \kappa(w_{12}) \alpha(w_{2})\\
	& + \sum_{w_1w_2w_3=w \atop \deg(w_i)>0} 
	   \sum_{w_{31}w_{32}w_{33}=w_3 \atop \deg(w_{32})>0,\deg(w_{33})>0} 
	\beta(w_{1}w_{31}w_{33}) \kappa(w_{32}) \alpha(w_{2})\\
	& + \sum_{w_1w_2w_3=w \atop \deg(w_i)>0} 
	   \sum_{w_{11}w_{12}w_{31}w_{32}=w_1w_3 \atop \deg(w_{ij})>0} 
	\beta(w_{11}w_{32}) \kappa(w_{12}w_{31}) \alpha(w_{2})\\
	& = \sum_{\pi \in \mathrm{NC}^{\scriptscriptstyle\mathrm{irr}}_3(m) \atop 
	{t(\pi) =\scalebox{0.5}{\Forest{[[][]]}}}} \beta(w_{\pi_1})(\kappa(w_{\pi_2})
	\alpha(w_{\pi_3})+\kappa(w_{\pi_3})\alpha(w_{\pi_2}))\\
	&+\sum_{\pi \in \mathrm{NC}^{\scriptscriptstyle\mathrm{irr}}_3(m) \atop {t(\pi) =\ell_3}} 
	\beta(a_{\pi_1})\kappa(a_{\pi_2})\alpha(a_{\pi_3}).
\end{align}
Note that for $\pi=\{\pi_1,\pi_2,\pi_3\} \in \mathrm{NC}^{\scriptscriptstyle\mathrm{irr}}_3(m)$ the tree $t(\pi) =\scalebox{0.5}{\Forest{[[][]]}}$ corresponds to the following nesting of blocks, $\pi_1\leq \pi_2,\pi_3$. For $\alpha=\beta=\kappa$ this gives, using the notation \eqref{trickynotation}:
$$
	L^2_{\kappa \scriptscriptstyle{\rhd}}(\kappa)(w)=\kappa \rhd (\kappa \rhd \kappa) (w)
	= \sum_{\pi \in \mathrm{NC}^{\scriptscriptstyle\mathrm{irr}}_3(m) \atop t(\pi) 
	=\scalebox{0.5}{\Forest{[[][]]}}} 2\kappa(a_\pi)
	+\sum_{\pi \in \mathrm{NC}^{\scriptscriptstyle\mathrm{irr}}_3(m) \atop t(\pi) =\ell_3} 
	\kappa(a_\pi).
$$

We compute the general left iterated pre-Lie products in $\mathfrak g$ using irreducible monotone non-crossing partitions. Then we apply this computation to cumulant-cumulant relations.

\begin{prop}\label{prop:leftiteration2}
Let $\kappa_1,\ldots,\kappa_{n+1}\in \mathfrak{g}$ and $w=a_1 \cdots a_m \in T_+(A)$ a word of length $\deg(w) =m \ge n+2$, $n>0$. Then 
\begin{equation}
\label{obs2bis}
	\kappa_1\rhd(\cdots \kappa_{n-1}\rhd(\kappa_n\rhd\kappa_{n+1})\cdots)(w)
	= \sum_{\pi \in \mathcal{M}^{\scriptscriptstyle\mathrm{irr}}_{n+1}(m)} 
	(-1)^{n} \kappa_{n+1}(a_{\pi_1})\cdots\kappa_{1}(a_{\pi_{n+1}}).
\end{equation}
\end{prop}

\begin{proof}
The proof goes by induction. Given a subset $S$ of $[m]$ with $1,m\in S$, we write $Int^{\scriptscriptstyle\mathrm{irr}}(S)$ for the intervals in $S$ that do neither contain $1$ nor $m$. Then,
\begin{align*}
	&\kappa_1\rhd(\cdots \kappa_{n-1}\rhd(\kappa_n\rhd\kappa_{n+1})\cdots)(w)\\
	&\quad = - \sum_{\pi_{n+1} \in Int^{\scriptscriptstyle\mathrm{irr}}([m])} \kappa_2 \rhd(\cdots \kappa_{n-1}
	\rhd(\kappa_n\rhd\kappa_{n+1})\cdots)(a_{[m]\backslash \pi_{n+1}})\kappa_1(a_{\pi_{n+1}}).
\end{align*}
By the induction hypothesis, 
$$
	\kappa_2\rhd(\cdots \kappa_{n-1}\rhd(\kappa_n\rhd\kappa_{n+1})\cdots)(a_{[m]\backslash\pi_{n+1}})
	= \sum_{\pi' \in \mathcal{M}^{\scriptscriptstyle\mathrm{irr}}_{n}([m]\backslash\pi_{n+1})}
	(-1)^{n-1} \kappa_{n+1}(a_{\pi'_1})\cdots\kappa_{2}(a_{\pi'_{n}}),
$$
where $\mathcal{M}^{\scriptscriptstyle\mathrm{irr}}_{n}([m]\backslash \pi_{n+1})$ stands for the set of monotone irreducible partitions of the set $[m]\backslash\pi_{n+1}$ with $n$ blocks. The proof follows since $\mathcal{M}^{\scriptscriptstyle\mathrm{irr}}_{n+1}(m)$ is in bijection with the pairs $(\pi_{n+1},\pi')$ appearing in the two summations.
\end{proof}

\begin{cor}\label{cor:leftiteration}
Let $\kappa$ and $\rho$ be linear forms on $T_+(A)$ and $w=a_1 \cdots a_m \in T_+(A)$ a word of length $\deg(w) =m \ge n+2$, $n>0$. Then 
\allowdisplaybreaks
\begin{align}
\label{obs2}	
\begin{split}
	L^n_{\rho \scriptscriptstyle{\rhd}}(\kappa)(w)
	&=  \sum_{\pi \in \mathrm{NC}^{\scriptscriptstyle\mathrm{irr}}_{n+1}(m)\atop 1,m\in\pi_1} 
	(-1)^{n} m(\pi)\kappa(w_{\pi_1})\rho(w_{\pi \backslash\pi_1}),
\end{split}
\end{align}
with the notation $\rho(w_{\pi \backslash\pi_1})=\rho(w_{\pi_2})\dots\rho(w_{\pi_{n+1}})$.
\end{cor}

\begin{rem}
Another verification of this corollary via induction can be given using \eqref{Kreimer} in Lemma \ref{lem:kreimer} together with the min-max order on non-crossing partitions mentioned in Remark \ref{rem:minmax}. Looking at a rooted tree as a poset, the coefficients in \eqref{obs2} count the number of total order extension of said poset. Looking at these coefficients from the viewpoint of irreducible non-crossing partitions, they count the number of total orderings of the blocks. 
\end{rem}

Recall now the free-monotone cumulant-cumulant relation \eqref{MagnusKey2}:
$$
	\kappa=W(\rho)=\rho+\sum\limits_{n>1}\frac{1}{(n+1)!}L^n_{\rho \scriptscriptstyle{\rhd}}(\rho).
$$
Applying identity (\ref{obs2}), we get another proof of the
\begin{theo}[\cite{arizmendi_15}]
For elements $a_1, \ldots, a_n \in A$ the multivariate free cumulants can be expressed in terms of multivariate monotone cumulants by:
\begin{equation}
\label{freemono1}
	\kappa(a_1\cdots a_n)
	= \sum_{\pi \in \mathrm{NC}^{\scriptscriptstyle\mathrm{irr}}(n)} 
	\frac{(-1)^{|\pi|-1}}{t(\pi)!} \rho(a_\pi).
\end{equation}
\end{theo}

In the Boolean case, where $\beta=-W(-\rho)$, we immediately find by an analogous computation
\begin{equation}
\label{boolmono1}
	\beta(a_1\cdots a_n)
	= \sum_{\pi \in \mathrm{NC}^{\scriptscriptstyle\mathrm{irr}}(n)} 
	\frac{1}{t(\pi)!} \rho(a_\pi).
\end{equation}

\section{Monotone-free cumulant-cumulant relations}
\label{sec:Fla2}

We adapt now  constructions by Murua on rooted trees \cite{murua_07} to non-crossing partitions. The definition of quasi-monotone partitions  below is also inspired independently by similar constructions on quasi-posets and finite topologies related to quasi-symmetric functions \cite{fmp_17}.

A quasi-order is a binary relation which is reflexive and transitive but not necessarily antisymmetric. A quasi-order $\preceq$ is total if two elements are always comparable, i.e., if $x \preceq y$ or $y \preceq x$ for all $x, y$. Two elements $x,y$ are equivalent ($x\equiv y$) for $\preceq$ if and only if $x\preceq y$ and $y\preceq x$. We write $x\prec y$ if $x\preceq y$ and $x$ and $y$ are not equivalent.
Total orders on a finite set $X$ of cardinality $n$ are in bijection with bijections from $X$ to $[n]$: given a total order, choose the unique increasing bijection from $X$ to $[n]$. Similarly, total quasi-orders on $X$ are in bijection with surjective maps from $X$ to $[k]$ with $k\leq n$ equal to the number of equivalence classes in $X$.

\begin{defi}
A quasi-monotone partition of $[n]$ is a non-crossing partition $\pi$ equipped with a total quasi-order $\preceq$ of its blocks compatible with the natural partial order $\leq$ in the sense that $\pi_i<\pi_j$ implies $\pi_i\prec \pi_j$. A quasi-monotone partition is of rank $k$ if $\preceq$ has $k$ equivalence classes. The set of quasi-monotone partitions of $[n]$ of rank $k$ is denoted $\mathrm{QM}_k(n)$.
\end{defi}

\begin{rem}
\label{rem.quasi.colored}
Choosing such a quasi-order amounts to decorating the blocks with elements of the set $[k]$ with $k\leq n$ in such a way that $\pi_i<\pi_j$ implies $f(\pi_i)<f(\pi _j)$, where $f$ stands for the (surjective) decoration map. The case where $f(\pi_i)\leq f(\pi _j)$ is required in the previous condition was considered in \cite{arizmendi_15} under the name of non-decreasing $k$-colored non-crossing partitions. 
\end{rem}

Let now $\pi$ be an irreducible non-crossing partition. We will write $\omega_k(\pi)$ for the number of these total quasi-orders on $\pi$, that is the number of quasi-monotone partitions of rank $k$ associated to the irreducible non-crossing partition $\pi$. As this number depends only on the order $\leq$, it depends only on the tree associated to $\pi$ and we write $\omega_k(t(\pi)):=\omega_k(\pi)$ (see \cite[Def.~12]{murua_07}).
 
We also define, following \cite[Thm.~10]{murua_07}, for $\pi$ an irreducible non-crossing partition of $[n]$,
\begin{equation}
\label{muruaomega}
	\omega(\pi)=\omega(t(\pi)):=\sum\limits_{k=1}^n\frac{(-1)^{k+1}}{k}\omega_k(\pi).
\end{equation}
For a general non-crossing partition $\pi'$ of $[n]$, we extend the previous definition by taking $\omega(\pi')$ to be  the product of the evaluations of $\omega$ on the irreducible components of $\pi'$. That is, $\pi'$ decomposes  as a union of irreducible non-crossing partitions, $\pi'=\rho_1\cup \rho_2\cup \cdots \cup \rho_l$ and  we set:
\begin{equation}
\label{ofla}
	\omega(\pi'):=\prod\limits_{i=1}^l\omega(\rho_i).
\end{equation}

These numbers appear in the analysis of the continuous Baker--Campbell--Hausdorff problem (to compute the logarithm of a flow) in a Hall basis as well as in the context of backward error analysis in numerical analysis \cite{calaque-etal_11,chv_05,chartier-etal_10,murua_07}. Their appearance is therefore natural in the context of the pre-Lie approach to cumulant-cumulant relations.

The first few terms of $\omega$ are listed below
$$
	\omega(\Forest{[]}\!)
		=-\frac{1}{2},\ \ 
	\omega(\scalebox{0.8}{\Forest{[[[]]]}}\!)	
		=\frac{1}{3},\ \ 
	\omega(\scalebox{0.8}{\Forest{[[][]]}}\!)	
		= - \frac{1}{2} + \frac{2}{3} = \frac{1}{6}.			
$$
The following table can be found in \cite{chartier-etal_10}.
\begin{equation}
\label{omegacoefftable1}
\begin{tabular} {c|c|c|c|c|c|c|c|c|c|c|c|c}
	$t$ & \scalebox{0.5}{\Forest{[]}} & \scalebox{0.5}{\Forest{[[]]}} & \scalebox{0.5}{\Forest{[[[]]]}} & \scalebox{0.5}{\Forest{[[][]]}} & \scalebox{0.5}{\Forest{[[[][]]]}} & \scalebox{0.5}{\Forest{[[[[]]]]}} 
	  & \scalebox{0.5}{\Forest{[[][[]]]}} & \scalebox{0.5}{\Forest{[[][][]]}} & \scalebox{0.5}{\Forest{[[][][][]]}} \\[0.3cm]
\hline	
	$\omega$ 	& $1$ & $-\frac{1}{2}$ & $\frac{1}{3}$ & $\frac{1}{6}$ & $- \frac{1}{6}$ & $-\frac{1}{4}$ & $-\frac{1}{12}$ & $0$ & $-\frac{1}{30}$ 
\end{tabular} 
\end{equation}

\begin{equation}
\label{omegacoefftable2}
\begin{tabular} {c|c|c|c|c|c|c|c|c|c|c|c|c}
	$t$ 	& \scalebox{0.5}{\Forest{[[[][][]]]}} & \scalebox{0.5}{\Forest{[[[]][[]]]}} & \scalebox{0.5}{\Forest{[[][[][]]]}} & \scalebox{0.5}{\Forest{[[[[[]]]]]}} & \scalebox{0.5}{\Forest{[[][[[]]]]}} & \scalebox{0.5}{\Forest{[[[[][]]]]}} & \scalebox{0.5}{\Forest{[[[][[]]]]}} 
	& \scalebox{0.5}{\Forest{[[][][[]]]}} \\[0.3cm]
\hline	
	$\omega$ & $\frac{1}{30}$ & $\frac{1}{30}$ & $\frac{1}{60}$ & $\frac{1}{5}$ 	& $\frac{1}{20}$ 	& $\frac{3}{20}$ 	& $\frac{1}{10}$ &
	$\frac{-1}{60}$ 
\end{tabular}
\end{equation}

\vspace{0.3cm}

Before entering cumulant-cumulant relations, let us introduce a last construction. Let $\pi=\{\pi_1,\dots,\pi_k\}$ be a non-crossing partition of $[n]$ and $V=\{\pi_{i_1},\dots,\pi_{i_l}\}$ a subset of $\pi$ including all the minimal elements of $\pi$ for the $\leq$-order. We write $Sub(\pi)$ for the set of all such subsets. The set of the $|V|=l$ blocks belonging to $V$ defines a non-crossing partition of $\bigcup_{j=1}^l\pi_{i_j}$ that we denote $\nu(V)$.

We call $V$-connected components of $\pi$ the sets $S_{\pi_{i_j}},\ j\leq l$, of blocks in $\pi$: 
$$
	S_{\pi_{i_j}}=\{\pi_m,\ m\leq k| \pi_m\geq \pi_{i_j} \ \textrm{and}\ \nexists j'\leq l\ 
	\textrm{with}\ \pi_m\geq \pi_{i_{j'}}>\pi_{i_j}\}.
$$
The sets of blocks $S_{\pi_{i_j}}$ are disjoint and define a partition of $\pi$. By construction, each set $S_{\pi_{i_j}}$ is an irreducible non-crossing partition of the set of integers $X_j:=\bigcup_{\pi_i\in S_{\pi_{i_j}}}\pi_i$ with minimal element $\pi_{i_j}$. We set:
\begin{equation}
\label{truc}
	V(\pi):=\{S_{\pi_{i_1}},\dots,S_{\pi_{i_l}}\},
\end{equation}
that we view as a family of irreducible non-crossing partitions with associated forest 
$$
	f(V(\pi))=t(S_{\pi_{i_1}})\cdots t(S_{\pi_{i_l}}).
$$

Notice, for further use, that there is a bijection between the pairs $(\pi,V)$ as above and the pairs $(\nu,(\alpha_i)_{1\leq i\leq l}))$, where $\nu=(\nu_1,\dots,\nu_l)$ is a non-crossing partition and $\alpha_i$ is an irreducible non-crossing partition of $\nu_i$. The bijection is given by 
\begin{equation}
\label{bij1}
	\nu_j:= X_j,\ \ \alpha_j:=S_{\pi_{i_j}}
\end{equation} 
with inverse 
\begin{equation}\label{bij2}
	\pi:=\bigsqcup_{j=1}^l\alpha_j,\ \ V=\{\alpha_1^1,\dots,\alpha_l^1\},
\end{equation} 
where $\alpha_j^1$ stands for the minimal block in $\alpha_j$.

\begin{rem}
For those readers familiar with the min-max order, the previous analysis can be described as follows. Given $\pi=\{\pi_1,\dots, \pi_k\}$, take a partition $\sigma=\{\sigma_1,\dots,\sigma_l\}$ such that $\sigma\gg \pi$. By definition of the min-max order, for every $\sigma_j\in\sigma$ there is a $\pi_{i_j}\in\pi$ such that $\min(\sigma_j),\max(\sigma_j)\in \pi_{i_j}$, this gives the set $V=\{\pi_{i_1},\dots,\pi_{i_l}\}$ (which is also the partition $\nu(V)$). Note that $V$ clearly contains all outer blocks of $\pi$. Then $V(\pi)=\{ \pi|_{\sigma_1},\dots, \pi|_{\sigma_1}\}$ is the set of irreducible partitions obtained when restricting $\pi$ to the blocks of $\sigma$ (this means that $X_j=\rho_j$ and $S_{i_j}=\pi|_{\sigma_j}$).
\end{rem}

The following Proposition translates \cite[Eq.~(41)]{murua_07} in the language of non-crossing partitions.

\begin{prop}\label{techprop}
Given an irreducible non-crossing partition $\pi=\{\pi_1,\dots,\pi_n\}$ with minimal element $\pi_1$, we write $\pi'$ for the non-crossing partition $\{\pi_2,\dots,\pi_n\}$ and
have:
\begin{equation}
\label{newmuruaomega}
	\omega(\pi)=\sum\limits_{V\in Sub(\pi')}\frac{B_{|V|}}{f(\nu(V))!}\omega(V(\pi')).
\end{equation}
\end{prop}

In this proposition, $f(\nu(V))$ is the forest associated to non-crossing partition $\nu(V)$. Recall that the factorial of a forest $f$ is the product of the tree factorials $t_i!$, where $t_i$ runs over the trees in the forest $f$. 

\medskip

We address now the central result of this work, i.e., a closed formula for multivariate monotone-free cumulant-cumulant relation. In \eqref{MagnusKey1} we saw that they are related in terms of the pre-Lie Magnus expansion
\begin{equation}
\label{MagnusRel1}
	\rho = \Omega'(\kappa).
\end{equation}
The objective is to express the evaluation of this formula on a word in terms of a sum over irreducible non-crossing partitions. An analogous approach applies to the monotone-Boolean relation, $\rho = -\Omega'(-\beta)$. 

Using the definition of the pre-Lie Magnus expansion \eqref{preLieMag1}, relation \eqref{MagnusRel1} expands into
\begin{equation}
\label{monofree}
	\rho = \sum_{n \ge 0} \frac{B_n}{n!} L^{n}_{\Omega(\kappa) \scriptscriptstyle \rhd}(\kappa).
\end{equation}
Computing up to order five gives the monotone-free cumulant-cumulant relations gives
\begin{align}
	\rho(a_1) 	
		&= \Omega'(\kappa) (a_1)= \kappa(a_1)	\label{rel1}\\[0.3cm]
	\rho(a_1a_2) 
		&=\Omega'(\kappa) (a_1a_2) 
		= \kappa(a_1a_2)\label{rel2}\\[0.3cm]
	\rho(a_1a_2a_3) 
		&=\Omega'(\kappa) (a_1a_2a_3) 
		= \kappa(a_1a_2a_3) 
				+ \frac{1}{2}\kappa(a_1a_3)\kappa(a_2) 		\label{rel3}\\[0.3cm] 
	 \begin{split}
	 \rho(a_1a_2a_3a_4) 
	 	&=  \Omega'(\kappa) (a_1a_2a_3a_4) 
		= \kappa(a_1a_2a_3a_4) 
			+ \frac{1}{2}\kappa(a_1a_4)\kappa(a_2a_3)\\
		&	+ \frac{1}{2}\kappa(a_1a_3a_4)\kappa(a_2) 
			+ \frac{1}{2}\kappa(a_1a_2a_4)\kappa(a_3)
			+ \frac{1}{6} \kappa(a_1a_4)\kappa(a_2)\kappa(a_3)			
	\label{rel4}
	\end{split}\\
	\begin{split}
	\rho(a_1a_2a_3a_4a_5) 
		&=\Omega'(\kappa) (a_1a_2a_3a_4a_5) \\
		&= \kappa(a_1a_2a_3a_4a_5) 
			+ \frac{1}{2}\kappa(a_1a_5)\kappa(a_2a_3a_4)
			+ \frac{1}{2}\kappa(a_1a_4a_5)\kappa(a_2a_3)\\
		&	+ \frac{1}{2}\kappa(a_1a_2a_5)\kappa(a_3a_4) 
			+ \frac{1}{2}\kappa(a_1a_3a_4a_5)\kappa(a_2)
			+ \frac{1}{2}\kappa(a_1a_2a_4a_5)\kappa(a_3)\\
		&	+ \frac{1}{2}\kappa(a_1a_2a_3a_5)\kappa(a_4)
			+ \frac{1}{6}\kappa(a_1a_4a_5)\kappa(a_2)\kappa(a_3)\\
		&	+ \frac{1}{6}\kappa(a_1a_3a_5)\kappa(a_2)\kappa(a_4) 
			+ \frac{1}{6}\kappa(a_1a_2a_5)\kappa(a_3)\kappa(a_4)\\
		&	 + \frac{1}{6}\kappa(a_1a_5)\kappa(a_2a_3)\kappa(a_4)
			+ \frac{1}{6}\kappa(a_1a_5)\kappa(a_3a_4)\kappa(a_2)
			+ \frac{1}{3} \kappa(a_1a_5)\kappa(a_2a_4)\kappa(a_3).	\label{rel5}
	\end{split}		 
\end{align}
In the relations \eqref{rel1}-\eqref{rel5}, the coefficients happen to depend only on the nesting structure of the corresponding irreducible non-crossing partitions. Moreover, they identify with the corresponding values of $\omega$:
\begin{align}
	\rho(a_1a_2)
	&=\sum_{\pi \in \mathrm{NC}^{\scriptscriptstyle\mathrm{irr}}(2)}(-1)^{|\pi|-1}\omega(t(\pi))\kappa(a_\pi)\\
	\rho(a_1a_2a_3)
	&=\sum_{\pi \in \mathrm{NC}^{\scriptscriptstyle\mathrm{irr}}(3)}(-1)^{|\pi|-1}\omega(t(\pi))\kappa(a_\pi)\\
	\rho(a_1 \cdots a_4)
	&=\sum_{\pi \in \mathrm{NC}^{\scriptscriptstyle\mathrm{irr}}(4)}(-1)^{|\pi|-1}\omega(t(\pi))\kappa(a_\pi)\\
	\rho(a_1 \cdots a_5)
	&=\sum_{\pi \in \mathrm{NC}^{\scriptscriptstyle\mathrm{irr}}(5)}(-1)^{|\pi|-1}\omega(t(\pi))\kappa(a_\pi),
\end{align}
with coefficients given by the first four entries in the table \eqref{omegacoefftable1}. These computations motivate Theorem \ref{thm:main} below.

\begin{rem}
Observe that the effective degree (Lemma \ref{effect}) determines the number of terms that are necessary to compute the cumulant-cumulant relation up to order $n$ in \eqref{monofree}.
\end{rem}

Proposition \ref{prop:bridges} together with eq.~(\ref{obs2}) are central to the proof of Theorem \ref{thm:main}, which addresses the problem of expressing monotone cumulants in terms of free (and Boolean) cumulants.

\begin{theo}\label{thm:main}
\begin{equation}
\label{treeMag3}
	\rho(a_1 \cdots a_m)
	=\sum_{\pi \in \mathrm{NC}^{\scriptscriptstyle\mathrm{irr}}(m)} 
	(-1)^{|\pi|-1}\omega(t(\pi)) \prod_{\pi_i \in \pi}\kappa(a_{\pi_i}).
\end{equation}
\end{theo}

\begin{proof} Notice first that the effective degree $\#(L^{n}_{\Omega(\kappa) \scriptscriptstyle{\rhd}}(\kappa))$ is at least $n+1$. Saying this, we find for a word $w=a_1 \cdots a_m \in T_+(A)$:
\begin{align*}
	\rho(w) 
	&= \Omega'(\kappa)(w) =\sum_{n \ge 0}^{m-1} \frac{B_n}{n!}L^n_{\Omega(\kappa) \scriptscriptstyle{\rhd}}(\kappa)(w) \\
	&= \kappa(w) + B_1 (\rho \rhd \kappa)(w) 
		+ \frac{B_2}{2!}(\rho \rhd (\rho \rhd \kappa))(w) 
		+ \cdots 
		+ \frac{B_{m-1}}{(m-1)!}L^{m-1}_{\Omega(\kappa) \scriptscriptstyle{\rhd}}(\kappa)(w)\\
	&=\sum_{n=0}^{m-1} \frac{B_n}{n!} (-1)^{n}
	    \sum_{\pi\in \mathrm{NC}^{\scriptscriptstyle\mathrm{irr}}_{n+1}(m) \atop 1,m \in \pi_1}
	   \frac{|t(\pi)|!}{t(\pi)!} \kappa(a_{\pi_1})\prod_{i=2}^{n+1} \rho(a_{\pi_i}) \\
	&=\sum_{n=0}^{m-1} \frac{B_n}{n!} (-1)^{n}
	    \sum_{\pi\in \mathrm{NC}^{\scriptscriptstyle\mathrm{irr}}_{n+1}(m) \atop 1,m \in \pi_1}
	   \frac{(n+1)!}{(n+1)f(\pi')!} \kappa(a_{\pi_1})\prod_{i=2}^{n+1} \rho(a_{\pi_i}),
\end{align*}	   
where $\pi'=\pi\backslash\pi_1=\{\pi_2,\pi_3,\dots,\pi_{n+1}\}$ is the non-crossing partition obtained by removing $\pi_1$ from $\pi$. We then obtain
\begin{align*}
\rho(w) 
	&=\sum_{n=0}^{m-1} (-1)^{n}
	\sum_{\pi \in \mathrm{NC}^{\scriptscriptstyle\mathrm{irr}}_{n+1}(m) \atop 1,m \in \pi_1}
	\frac{B_{|\pi'|}}{f(\pi')!}  \kappa(a_{\pi_1})\prod_{i=2}^{n+1} \pai 
	\sum_{\sigma_i \in \mathrm{NC}^{\scriptscriptstyle\mathrm{irr}}(\pi_i)} (-1)^{|\sigma_i|-1}
	\omega(t(\sigma_i))\kappa(a_{\sigma_i})\pad
\end{align*}
and, using the bijection (\ref{bij1},\ref{bij2}),
\begin{align*}
\rho(w)
    &=\sum_{\mu \in \mathrm{NC}^{\scriptscriptstyle\mathrm{irr}}(m)}(-1)^{|\nu|-1}\kappa(a_{\mu})\sum\limits_{V\in Sub(\mu')}
	\frac{B_{|V|}}{f(\nu(V))!} \omega(V(\mu')) 
\end{align*}
so that, by Proposition \ref{techprop}
\begin{align*}
\rho(w)
    &=\sum_{\mu \in \mathrm{NC}^{\scriptscriptstyle\mathrm{irr}}(m)}(-1)^{|\mu|-1}\omega(t(\mu))\kappa(a_{\mu}),
\end{align*}
and the statement follows.
\end{proof}



\begin{thebibliography}{mmm}


\bibitem{AG_81}
A.~Agrachev, R.~Gamkrelidze, 
{\emph{Chronological algebras and nonstationary vector fields}}, 
 J.~Sov.~Math.~{\bf{17}}, (1981) 1650-1675.

\bibitem{AS_04}
A.~Agrachev, Y.~Sachkov,
{\it{Control Theory from the Geometric Viewpoint}},
Encyclopaedia of Mathematical Sciences {\bf{84}}, Springer-Verlag Berlin Heidelberg 2004.

\bibitem{arizmendi_15}
O.~Arizmendi, T.~Hasebe, F.~Lehner, C.~Vargas,
{\textsl{Relations between cumulants in noncommutative probability}},
Adv. Math.~{\bf{282}}, (2015) 56-92.

\bibitem{BN_08}
S.~T. Belinschi, A.~Nica,
{\textsl{$\eta$-series and a Boolean Bercovici--Pata bijection for bounded $k$-tuples}},
Adv. Math.~{\bf{217}}, (2008) 1-41.

\bibitem{blanes_09}
S.~Blanes, F.~Casas, J.A.~Oteo, J.~Ros,
{\textsl{Magnus expansion: mathematical study and physical applications}},
Phys.~Rep.~{\bf{470}}, (2009) 151-238.

\bibitem{burde_06}
D.~Burde,
{\textsl{Left-symmetric algebras, or pre-Lie algebras in geometry and physics}},
Cent. Eur. J. Math.~{\bf{4}} Nr. 3, (2006) 323-357.

\bibitem{calaque-etal_11}
D. Calaque, K.~Ebrahimi-Fard, D. Manchon,
{\textsl{Two interacting Hopf algebras of trees: A Hopf-algebraic approach to composition and substitution of B-series}}, 
Adv.~App.~Math.~{\bf{47}}, (2011) 282-308.

\bibitem{cartier_11}
P.~Cartier,
{\textsl{Vinberg algebras, Lie groups and combinatorics}},
Clay Mathematical Proceedings {\bf{11}}, (2011) 107-126. 

\bibitem{chapat_13}
F. Chapoton, F. Patras, 
{\textsl{Enveloping algebras of preLie algebras, Solomon idempotents and the Magnus formula}}, 
Int. J. Algebra and Computation {\bf{23}}, No. 4, (2013) 853-861.

\bibitem{chv_05}
Ph. Chartier, E. Hairer, G. Vilmart, 
{\emph{A substitution law for B-series vector fields}}, 
preprint INRIA No.~5498 (2005).

\bibitem{chartier-etal_10}
Ph. Chartier, E. Hairer, G. Vilmart, 
{\emph{Algebraic structures of B-series}}, 
Found. Comput. Math.~{\bf{10}}, (2010) 407-427.

\bibitem{ebrahimimanchon_09Mag}
K.~Ebrahimi-Fard, D.~Manchon,
{\emph{A Magnus- and Fer-type formula in dendriform algebras}}, 
Found. Comput. Math.~{\bf{9}}, (2009) 295-316.

\bibitem{ebrahimimanchon_09}
K.~Ebrahimi-Fard, D.~Manchon,
{\emph{Dendriform Equations}}, 
J.~Algebra {\bf{322}}, (2009) 4053-4079.

\bibitem{ebrahimipatras_14}
K.~Ebrahimi-Fard, F.~Patras,
{\emph{The Pre-Lie Structure of the Time-Ordered Exponential}},
Lett.~Math.~Phys.~{\bf{104}}, (2014) 1281-1302.

\bibitem{ebrahimipatras_15}
K.~Ebrahimi-Fard, F.~Patras,
{\emph{Cumulants, free cumulants and half-shuffles}}, 
Proc.~R.~Soc.~A {\bf{471}}, 2176, (2015).

\bibitem{ebrahimipatras_16}
K.~Ebrahimi-Fard, F.~Patras,
{\emph{The splitting process in free probability theory}},
Int.~Math.~Res.~Not.~{\bf 9}, (2016) 2647-2676.

\bibitem{ebrahimipatras_18}
K. Ebrahimi-Fard, F. Patras, 
{\emph{Monotone, free, and boolean cumulants from a Hopf algebraic point of view}}, 
Adv. Math.~{\bf{328}}, (2018) 112-132.

\bibitem{ebrahimipatras_19}
K. Ebrahimi-Fard, F. Patras, 
{\emph{Shuffle group laws. Applications in free probability}},
P.~Lond.~Math.~Soc.~{\bf{119}}, (2019) 814-840. 

\bibitem{ebrahimipatras_20}
K. Ebrahimi-Fard, F. Patras, 
{\emph{A group-theoretical approach to conditionally free cumulants}},
in "Algebraic Combinatorics, Resurgence, Moulds and Applications (CARMA) Vol.~1", IRMA Lectures in Mathematics and Theoretical Physics {\bf{31}}, 2020.

\bibitem{Nacer}
K.~Ebrahimi-Fard, F.~Patras,
{\emph{From iterated integrals and chronological calculus to Hopf and Rota--Baxter algebras}}, 
Encyclopedia in Algebra and Applications (to appear) arXiv:1911.08766.

\bibitem{fmp_17}
L.~Foissy, C.~Malvenuto, F.~Patras, 
{\emph{A theory of pictures for quasi-posets}}, 
J. Algebra {\bf{477}}, (2017) 496-515.

\bibitem{hasebesaigo_11}
T.~Hasebe, H.~Saigo,
{\emph{The monotone cumulants}}, 
Annales de l'Institut Henri Poincar\'e - Probabilit\'es et Statistiques {\bf{47}}, No.~4, (2011) 1160-1170.
	
\bibitem{iserles_00}
A.~Iserles, H.~Z.~Munthe-Kaas, S.~P.~N{\o}rsett and A.~Zanna,
{\emph{Lie-group methods}},
Acta Numer.~{\bf{9}}, (2000) 215-365.

\bibitem{iserles_02}
A.~Iserles,
{\emph{Expansions That Grow on Trees}},
Notices of the AMS {\bf{49}}, No.~4, (2002) 430-440.

\bibitem{kreimer_00}
D.~Kreimer,
{\it{Chen's iterated integral represents the operator product expansion}},
Adv. Theor. Math. Phys.~{\bf{3}}, (2000) 627-670.
	
\bibitem{magnus_54}
W.~Magnus,
{\textsl{On the exponential solution of differential equations for a linear operator}},
Commun.~Pure Appl.~Math.~{\bf{7}}, (1954) 649-673.

\bibitem{manchon_11}
D.~Manchon,
{\emph{A short survey on pre-Lie algebras}}, 
E.~Schr\"odinger Institut Lectures in Math. Phys., ``Noncommutative Geometry and Physics: 
Renormalisation, Motives, Index Theory", Eur. Math. Soc, A.~Carey Ed.~(2011).

\bibitem{murua_07}
A.~Murua,
{\emph{The Hopf algebra of rooted trees, free Lie algebras, and Lie series}}, 
Found. Comput. Math.~{\bf{6}}, (2006) 387-426.

\bibitem{nicaspeicher_06}
A.~Nica, R.~Speicher, 
Lectures on the combinatorics of free probability,
London Mathematical Society Lecture Note Series, {\bf{335}} Cambridge University Press (2006).

\bibitem{speicher_97c} 
R.~Speicher, R.~Woroudi,
{\emph{Boolean convolution}},
In: Voiculescu, D. V. (ed.) Free Probability Theory. Proceedings, Toronto, Canada 1995, 
Fields Inst.~Commun.~{\bf{12}}, Providence, RI: Amer.~Math.~Soc., (1997) 267-279.

\end{thebibliography}
\end{document}